\documentclass{amsart}

\usepackage{amsmath}
\usepackage{amssymb}
\usepackage{amsthm}
\usepackage{geometry}               
\usepackage{tikz}
\usetikzlibrary{decorations.pathmorphing}
\usepackage{mathrsfs}
\usepackage{url}
\usepackage{algorithm}
\usepackage[noend]{algpseudocode}

\newtheorem{thm}{Theorem}
\newtheorem{prop}[thm]{Proposition}
\newtheorem{lemma}[thm]{Lemma}
\newtheorem{cor}[thm]{Corollary}

\newtheorem{Example}[thm]{Example}
\newenvironment{example}[1][]{\begin{Example}[#1]\rm}{\end{Example}}

\theoremstyle{definition}
\newtheorem{defn}[thm]{Definition}

\newtheorem{rmk}[thm]{Remark}

\newcommand{\spt}{\mathsf{SPT}(G)}
\newcommand{\gpf}{\mathsf{PF}(G)}
\newcommand{\dfs}{\mathrm{DFS}}

\newcommand{\NN}{\mathbb{N}}
\newcommand{\bv}{\texttt{burnt\_vertices}}
\newcommand{\damp}{\texttt{dampened\_edges}}
\newcommand{\tree}{\texttt{tree\_edges}}
\newcommand{\pf}{\mathcal{P}}

\title{$G$-parking functions and tree inversions}
\author{David Perkinson}
\email{davidp@reed.edu}
\author{Qiaoyu Yang}
\email{yangq@reed.edu}
\author{Kuai Yu}
\email{kyu@reed.edu}
\address{Reed College, Portland OR, 97202}

\begin{document}

\begin{abstract} A depth-first search version of Dhar's burning algorithm is
  used to give a bijection between the parking functions of a graph and labeled
  spanning trees, relating the degree of the parking function with the number of
  inversions of the spanning tree.  Specializing to the complete graph 
  solves a problem posed by R.\ Stanley.
\end{abstract}

\maketitle

\section{Introduction}
Let $G=(V,E)$ be a connected simple graph with vertex set $V=\{0,\dots,n\}$ and
edge set $E$.  Fix a {\em root} vertex $r\in V$ and let $\spt$ denote the set of
spanning trees of $G$ rooted at $r$.  We think of each element of $\spt$ as a
directed graph in which all paths lead away from the root.  If $i,j\in V$ and
$i$ lies on the unique path from $r$ to $j$ in the rooted spanning tree $T$,
then $i$ is an {\em ancestor} of~$j$ and $j$ is a {\em descendant} of~$i$ in
$T$.  If, in addition, there are no vertices between $i$ and $j$ on the path
from the root, then $i$ is the {\em parent} of its {\em child}~$j$, and $(i,j)$
is a directed edge of $T$.   

\begin{defn} An {\em inversion} of $T\in\spt$ is a pair of vertices $(i,j)$,
  such that $i$ is an ancestor of~$j$ in $T$ and $i>j$.  It is a {\em
  $\kappa$-inversion} if, in addition, $i$ is not the root and $i$'s parent is
  adjacent to~$j$ in $G$.  The number of $\kappa$-inversions of $T$ is the
  tree's  {\em $\kappa$-number}, denoted $\kappa(G,T)$.
\end{defn}
\begin{defn}\label{parking-function}
  A {\em parking function} for $G$ (with respect to the vertex $r$) is a
  function 
  \[
    \pf\colon V\setminus\{r\}\to \NN
  \]
  such that for every nonempty set $S \subseteq V\setminus\{r\}$, there exists
  $i\in S$ such that $\pf(i)<\deg_{S^c}(i)$, where $\deg_{S^c}(i)$ is the number
  of edges $\{i,j\}$ of $G$ with $j\notin S$ (including the possibility $j=r$).
  The {\em degree} of a parking function $\pf$ is
\[
  \deg\pf := \sum_{i\in V\setminus\{r\}}\pf(i).
\]
  
  The set of parking functions for $G$ is denoted by $\gpf$.  
\end{defn} 

In this work, we introduce the {\em $\dfs$-burning algorithm}. It is a melding 
of depth-first search with Dhar's burning algorithm~\cite{Dhar} from the Abelian sandpile
model, assigning a spanning tree to each parking function for
$G$.  Our main result is:
\begin{thm}\label{main} The DFS-burning algorithm (Algorithm~\ref{dfs-burning}) gives a
  bijection $\phi\colon\gpf\to\spt$ such that
  \[
    \kappa(G,\phi(\pf)) = g - \deg{\pf}
  \]
  where $g:=|E|-|V|+1$ is the {\em circuit rank}\footnote{In the theory of
    parking functions, the circuit rank is often called the genus due to
    the role it plays in the Riemann-Roch theorem for graphs,
    \cite{BakerNorine}.} of $G$.  The inverse to $\phi$ is given by
    Algorithm~\ref{inverse}.
\end{thm}
The reader is encouraged to refer to Figure~\ref{house graph} for an example.  

\subsection{Background.}
Parking functions were originally defined in a study of hashing techniques in
computer science,~\cite{Konheim}, phrased in terms of a problem involving
preferences of drivers for parking spaces.  Implicit in that definition is the
restriction to complete graphs.  Parking functions for general graphs have
appeared in a variety of guises: in the Riemann-Roch theory for graphs, parking
functions are known as {\em reduced divisors},~\cite{BakerNorine}; in the
context of chip-firing games, they are known as {\em superstable}
configurations,~\cite{Holroyd}.  Superstables are directly related to the {\em
recurrent} configurations for Dhar's Abelian sandpile model on the
graph,~\cite[Thm.~4.4]{Holroyd},~\cite{Dhar}, and thus to the set of {\em
critical} configurations in Biggs' dollar game,~\cite{Biggs}, and they serve as
representatives for Lorenzini's {\em group of
components},~\cite{Lorenzini},~\cite{Lorenzini2}.  The name ``$G$-parking
function'' was introduced by Postnikov and Shapiro in~\cite{Postnikov}. Parking
functions in the case of a complete graph have appeared in the theory of
symmetric functions,~\cite{Haiman},~\cite{Novelli}, and hyperplane
arrangements,~\cite{Stanley} (for the latter, see~\cite{Hopkins} for an
extension to more general graphs).

The {\em Tutte polynomial} for a simple graph $G$ is
\begin{equation*}
  T(G,x,y)= \sum_{A\subseteq E}(x-1)^{c(A)-c(E)}(y-1)^{c(A)+|A|-|V|}
\end{equation*}
where $c(A)$ is the number of connected components of the subgraph of $G$ with
vertex set $V$ and edge set $A$.  (For our purposes, we assume $G$ is
connected, so  $c(E)=1$.) Translating the work of Merino,~\cite{Merino}, into
the language of parking functions, $T(1,y)=\sum_{i=0}^{g}a_i\,y^i$ where
$g-a_i$ is the number of parking functions for $G$ of degree $d$ and $g$ is the
circuit rank of $G$.  (Hence, $y^g\,T(1,1/y)$ is the generating function for
the parking functions by degree.)  The definition of a $\kappa$-inversion is
due to Gessel,~\cite{Gessel:Decom}, where it is introduced for the purpose of
showing that $a_i$ is the number of spanning trees of $G$ with $\kappa$-number
$i$.  Theorem~\ref{main} may be regarded as an explanation of the coincidence. 

Inversions and $\kappa$-inversions are the same when $G$ is a complete graph
(or, more generally, if $G$ is a properly labeled {\em threshold graph} as
explained in Section~\ref{threshold graphs}).  In earlier work,
Kreweras,~\cite{Kreweras}, had already noticed that for a complete graph the
number of trees with inversion number $a$ equals the number of parking functions
of degree $g-a$.  Stanley,~\cite{Stanley}, presents this result and poses the
problem of finding a corresponding explicit bijection for the complete graph,
$K_n$, that does not depend on recursing through bijections for $K_i$ for $i<
n$,~\cite[Chapter 6, Exercise 4]{Stanley}.  This problem was the motivation for
our work and is generalized and solved by Theorem~\ref{main}.  Note that
although the algorithms of Theorem~\ref{main} use recursion, they recurse only
through the vertices of a fixed graph.  The runtime is $O(|V|+|E|)$, as it is
for the usual depth-first search of a graph (and it would be a standard exercise
to rewrite the algorithm using a stack and avoiding recursion without changing
the runtime).

We were influenced by~\cite{BakerShokrieh}, which gives an exposition of the
work of Cori and Le Borne in~\cite{ClB}.  They describe an algorithm that gives a
bijection between parking functions and spanning trees in which a parking
function of degree $d$ is assigned a tree with external activity $g-d$.  We are
indebted to Farbod Shokrieh for explaining this work to us at the American
Institute of Mathematics workshop on {\em Generalizations of chip-firing and the
critical group}, July 2013.  See~\cite{Beissinger} for a recursively defined
permutation of labeled trees relating external activity to inversions.
While preparing this manuscript we became aware of the work of
Shin,~\cite{Shin}, later subsumed in a paper by de Oliveira and Las
Vergnas,~\cite{deOliveira}, in which Stanley's problem had previously been
solved.  Roughly, depth-first search is used to give a bijection between
permutations and trees with no inversions.  On the complete graph, permutations
may be thought of as maximal-degree parking functions.  To extend the bijection
to arbitrary parking functions on the complete graph, a procedure is given for
relabeling.  The advantages of the bijection of Theorem~\ref{main} are (i) it
applies to arbitrary (simple, connected, labeled) graphs, not only to complete
graphs, and (ii) the algorithms providing the bijection and its inverse are less
complicated, using little more than depth-first search, avoiding re-writing
rules.  We note that restricting our bijection to the case of complete graphs
gives a substantially different bijection from that of~\cite{deOliveira}
or~\cite{Shin}. 

Gessel and Sagan, \cite{GesselSagan}, defines the {\em neighbors-first search
(NFS)} of a graph and characterizes $\kappa$-inversions of a spanning tree in
terms of edges that are {\em externally active with respect to NFS},
(Theorem~6.4,~\cite{GesselSagan}).  The authors use NFS to describe a bijection
between spanning trees and parking functions on complete graphs.  The
bijection relates the number of $\kappa$-inversions to a statistic on parking
functions---thought of as hash functions---measuring the number of ``probes''
performed by a corresponding search protocol.

\subsection{Organization.} In Section~\ref{main section}, we describe
and verify the algorithms providing the bijection, then prove
Theorem~\ref{main}.  Section~\ref{threshold graphs} considers {\em threshold
graphs}, a class of graphs including the complete graphs.
Proposition~\ref{threshold-inversion} in that section shows that for
spanning trees of (suitably labeled) threshold graphs, every inversion is a
$\kappa$-inversion.

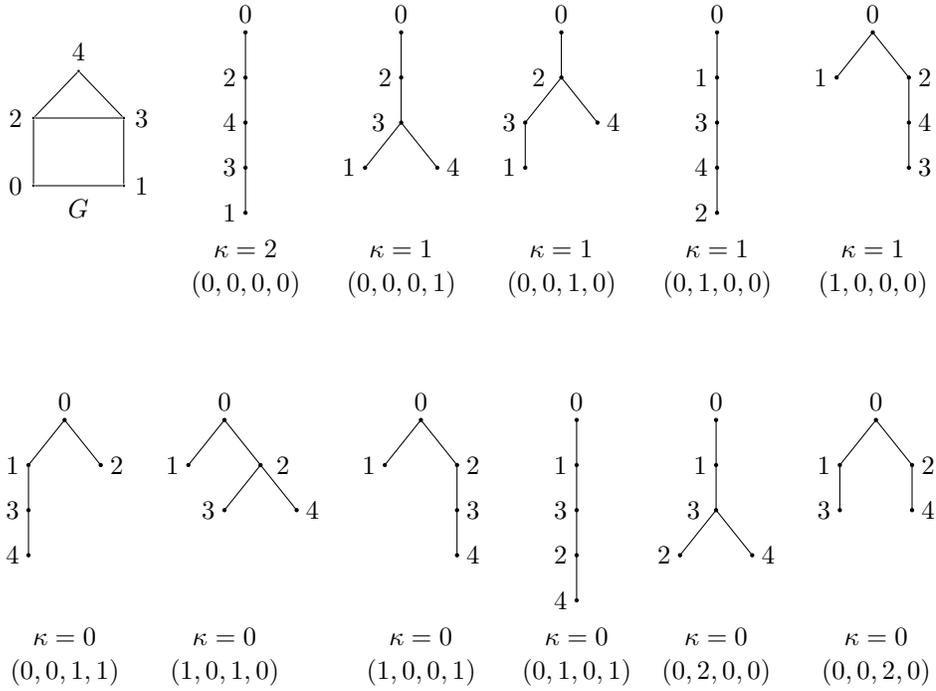
\begin{figure}[ht]
\begin{tikzpicture}[my_node/.style={fill, circle, inner sep=0pt,minimum size=1pt}, scale=0.6]
  \def\rad{1pt}
  \begin{scope}[shift={(0,3)},scale=0.5]
\node[my_node,above,label=above:$4$] (v4) at (0,5) {};
\node[my_node,label=right:$3$] (v3) at (2,3) {};
\node[my_node,label=left:$2$] (v2) at (-2,3) {};
\node[my_node,label=right:$1$] (v1) at (2,0) {};
\node[my_node,label=left:$0$] (v0) at (-2,0) {};
\draw (v0)--(v1)--(v3)--(v2)--(v4)--(v3);
\draw (v0)--(v2);
\draw (0,-1) node {$G$};
\draw (0,-5) node {};
\end{scope}
\end{tikzpicture}
\hspace{1mm}
\begin{tikzpicture}[scale=0.6]
  \def\rad{1pt}
  \begin{scope}[shift={(0,0)}]
  \draw[fill] (0,0) circle [radius=\rad];
  \draw[fill] (0,-1) circle [radius=\rad];
  \draw[fill] (0,-2) circle [radius=\rad];
  \draw[fill] (0,-3) circle [radius=\rad];
  \draw[fill] (0,-4) circle [radius=\rad];
  \draw (0,0) node[above]{$0$} --
        (0,-1) node[left]{$2$} --
        (0,-2) node[left]{$4$} --
        (0,-3) node[left]{$3$} --
        (0,-4) node[left]{$1$};
  \node at (0,-4.8){$\kappa=2$};
  \node  at (0,-5.6){$(0,0,0,0)$};
  \end{scope}
\end{tikzpicture}
\hspace{1mm}
\begin{tikzpicture}[scale=0.6]
  \def\rad{1pt}
  \begin{scope}[shift={(0,0)}]
  \draw[fill] (0,0) circle [radius=\rad];
  \draw[fill] (0,-1) circle [radius=\rad];
  \draw[fill] (0,-2) circle [radius=\rad];
  \draw[fill] (-0.8,-3) circle [radius=\rad];
  \draw[fill] (0.8,-3) circle [radius=\rad];
  \draw	(-0.5,-2) node {$3$};
  \draw (0,0) node[above]{$0$} --
        (0,-1) node[left]{$2$} --
	(0,-2)  --
        (-0.8,-3) node[left]{$1$}
        (0,-2)--(0.8,-3) node[right]{$4$};
  \node at (0,-4.8){$\kappa=1$};
  \node  at (0,-5.6){$(0,0,0,1)$};
  \end{scope}
\end{tikzpicture}
\hspace{1mm}
\begin{tikzpicture}[scale=0.6]
  \def\rad{1pt}
  \begin{scope}[shift={(0,0)}]
  \draw[fill] (0,0) circle [radius=\rad];
  \draw[fill] (0,-1) circle [radius=\rad];
  \draw[fill] (-0.8,-2) circle [radius=\rad];
  \draw[fill] (-0.8,-3) circle [radius=\rad];
  \draw[fill] (0.8,-2) circle [radius=\rad];
  \draw (-0.5,-1) node {$2$};
  \draw (0,0) node[above]{$0$} -- (0,-1) 
  -- (0.8,-2) node[right]{$4$} (0,-1) 
  -- (-0.8,-2) node[left]{$3$} (-0.8,-2) 
  -- (-0.8,-3) node[left]{$1$}; \node at (0,-4.8){$\kappa=1$};
  \node  at (0,-5.6){$(0,0,1,0)$};
  \end{scope}
\end{tikzpicture}
\hspace{1 mm}
\begin{tikzpicture}[scale=0.6]
  \def\rad{1pt}
  \begin{scope}[shift={(0,0)}]
  \draw[fill] (0,0) circle [radius=\rad];
  \draw[fill] (0,-1) circle [radius=\rad];
  \draw[fill] (0,-2) circle [radius=\rad];
  \draw[fill] (0,-3) circle [radius=\rad];
  \draw[fill] (0,-4) circle [radius=\rad];
  \draw (0,0) node[above]{$0$} --
        (0,-1) node[left]{$1$} --
        (0,-2) node[left]{$3$} --
        (0,-3) node[left]{$4$} --
        (0,-4) node[left]{$2$};
  \node at (0,-4.8){$\kappa=1$};
  \node  at (0,-5.6){$(0,1,0,0)$};
  \end{scope}
\end{tikzpicture}
\hspace{1 mm}
\begin{tikzpicture}[scale=0.6]
  \def\rad{1pt}
  \begin{scope}[shift={(0,0)}]
  \draw[fill] (0,0) circle [radius=\rad];
  \draw[fill] (-0.8,-1) circle [radius=\rad];
  \draw[fill] (0.8,-1) circle [radius=\rad];
  \draw[fill] (0.8,-2) circle [radius=\rad];
  \draw[fill] (0.8,-3) circle [radius=\rad];
  \draw (0,0) node[above]{$0$} --
        (0.8,-1) node[right]{$2$} -- (0.8,-2) node[right]{$4$} -- (0.8,-3) node[right]{$3$}
        (0,0)  -- (-0.8,-1) node[left]{$1$};
  \node at (0,-4.8){$\kappa=1$};
  \node  at (0,-5.6){$(1,0,0,0)$};
  \end{scope}
\end{tikzpicture}
\vskip10mm
\begin{tikzpicture}[scale=0.6]
  \def\rad{1pt}
  \begin{scope}[shift={(0,0)}]
  \draw[fill] (0,0) circle [radius=\rad];
  \draw[fill] (-0.8,-1) circle [radius=\rad];
  \draw[fill] (-0.8,-2) circle [radius=\rad];
  \draw[fill] (-0.8,-3) circle [radius=\rad];
  \draw[fill] (0.8,-1) circle [radius=\rad];
  \draw (0,0) node[above]{$0$} --
        (-0.8,-1) node[left]{$1$} --
        (-0.8,-2) node[left]{$3$} --
        (-0.8,-3) node[left]{$4$}
        (0,0)  -- (0.8,-1) node[right]{$2$};
  \node at (0,-4.8){$\kappa=0$};
  \node  at (0,-5.6){$(0,0,1,1)$};
  \end{scope}
\end{tikzpicture}
\hspace{1 mm}
\begin{tikzpicture}[scale=0.6]
  \def\rad{1pt}
  \begin{scope}[shift={(0,0)}]
  \draw[fill] (0,0) circle [radius=\rad];
  \draw[fill] (-0.8,-1) circle [radius=\rad];
  \draw[fill] (0.8,-1) circle [radius=\rad];
  \draw[fill] (0,-2) circle [radius=\rad];
  \draw[fill] (1.6,-2) circle [radius=\rad];
  \draw (1.3,-1) node {$2$};
  \draw (0,0) node[above]{$0$} -- (-0.8,-1) node[left]{$1$}
  (0,0)  -- (0.8,-1)
	-- (0,-2) node[left]{$3$}
        (0.8,-1)  -- (1.6,-2) node[right]{$4$};
  \node at (0,-4.8){$\kappa=0$};
  \node  at (0,-5.6){$(1,0,1,0)$};
  \end{scope}
\end{tikzpicture}
\hspace{1 mm}
\begin{tikzpicture}[scale=0.6]
  \def\rad{1pt}
  \begin{scope}[shift={(0,0)}]
  \draw[fill] (0,0) circle [radius=\rad];
  \draw[fill] (-0.8,-1) circle [radius=\rad];
  \draw[fill] (0.8,-1) circle [radius=\rad];
  \draw[fill] (0.8,-2) circle [radius=\rad];
  \draw[fill] (0.8,-3) circle [radius=\rad];
  \draw (0,0) node[above]{$0$} -- (-0.8,-1) node[left]{$1$}
        (0,0) -- (0.8,-1) node[right]{$2$} --
        (0.8,-2) node[right]{$3$} -- (0.8,-3) node[right]{$4$};
  \node at (0,-4.8){$\kappa=0$};
  \node  at (0,-5.6){$(1,0,0,1)$};
  \end{scope}
\end{tikzpicture}
\hspace{1 mm}
\begin{tikzpicture}[scale=0.6]
  \def\rad{1pt}
  \begin{scope}[shift={(0,0)}]
  \draw[fill] (0,0) circle [radius=\rad];
  \draw[fill] (0,-1) circle [radius=\rad];
  \draw[fill] (0,-2) circle [radius=\rad];
  \draw[fill] (0,-3) circle [radius=\rad];
  \draw[fill] (0,-4) circle [radius=\rad];
  \draw (0,0) node[above]{$0$} --
        (0,-1) node[left]{$1$} --
        (0,-2) node[left]{$3$} --
        (0,-3) node[left]{$2$} --
        (0,-4) node[left]{$4$};
  \node at (0,-4.8){$\kappa=0$};
  \node  at (0,-5.6){$(0,1,0,1)$};
  \end{scope}
\end{tikzpicture}
\begin{tikzpicture}[scale=0.6]
  \def\rad{1pt}
  \begin{scope}[shift={(0,0)}]
  \draw[fill] (0,0) circle [radius=\rad];
  \draw[fill] (0,-1) circle [radius=\rad];
  \draw[fill] (0,-2) circle [radius=\rad];
  \draw[fill] (-0.8,-3) circle [radius=\rad];
  \draw[fill] (0.8,-3) circle [radius=\rad];
  \draw (-0.5,-2) node {$3$};
  \draw (0,0) node[above]{$0$} -- (0,-1) node[left]{$1$} -- 
  (0,-2) 
  -- (-0.8,-3) node[left]{$2$}
        (0,-2)  -- (0.8,-3) node[right]{$4$};
  \node at (0,-4.8){$\kappa=0$};
  \node  at (0,-5.6){$(0,2,0,0)$};
  \end{scope}
\end{tikzpicture}
\hspace{1 mm}
\begin{tikzpicture}[scale=0.6]
  \def\rad{1pt}
  \begin{scope}[shift={(0,0)}]
  \draw[fill] (0,0) circle [radius=\rad];
  \draw[fill] (-0.8,-1) circle [radius=\rad];
  \draw[fill] (-0.8,-2) circle [radius=\rad];
  \draw[fill] (0.8,-1) circle [radius=\rad];
  \draw[fill] (0.8,-2) circle [radius=\rad];
  \draw (0,0) node[above]{$0$} -- (-0.8,-1) node[left]{$1$} -- (-0.8,-2) node[left]{$3$}
        (0,0)  -- (0.8,-1) node[right]{$2$} -- (0.8,-2) node[right]{$4$};
  \node at (0,-4.8){$\kappa=0$};
  \node  at (0,-5.6){$(0,0,2,0)$};
  \end{scope}
\end{tikzpicture}
\caption{The graph $G$ with its $11$ spanning trees and their corresponding
  $\kappa$-numbers and parking functions (as provided by the $\dfs$-burning
  algorithm).  The root vertex is $r=0$.  Each parking function $\pf$ is written
  as vector with $i$-th component $\pf(i)$.}\label{house graph}
\end{figure}
\medskip

\noindent{\em Acknowledgements.}  The authors would like to thank Jim Fix and
Farbod Shokrieh for helpful conversations.  We thank Collin Perkinson for his
comments.  We also thank our anonymous referees for their thoughtful remarks.

\section{Proof of main theorem}\label{main section}

In this section, it is assumed that $G$ is a connected simple graph with
vertices $V=\{0,\dots,n\}$ and fixed root vertex $r\in V$.  We begin by
describing the $\dfs$-burning algorithm (where $\dfs$ stands for ``depth-first
search'').  The idea is to imagine that a fire is started at the root vertex,
$r$, of $G$ and spreads according to a depth-first rule, to be described below,
along the edges until all vertices are burnt.  A nonnegative function $\pf$ is
thought of as an allocation of drops of water to each non-root
vertex\footnote{In~\cite{BakerShokrieh}, ``firefighters'' play the role of our
drops of water.}.  Suppose fire travels along an edge $e$ to a vertex $v$.  If
there are still drops of water on $v$, one drop will be used to ``dampen'' $e$,
thus protecting $v$ from the flame, and the search backtracks.  The depth-first
rule is as follows: if there is no remaining water at $v$, then $e$ is
marked,~$v$ is burnt, and the fire proceeds from $v$ along an edge {\em to the
  largest unburnt neighboring vertex}.  In the end, either (i) all vertices are
  burnt, $\pf$ is a parking function, and the collection of marked edges forms a
  spanning tree, $T_{\pf}$, (and the number of dampened edges is $\deg\pf$) or
  (ii) the nonempty set of unburnt vertices,~$S$, serves as a certificate that
  $\pf$ is not a parking function: $\pf(j)\geq\deg_{S^c}(j)$ for each $j\in S$.
  Algorithm~\ref{dfs-burning} provides a precise statement of the
  $\dfs$-algorithm.  Its validity is established in Theorem~\ref{dfs-bijection}.

\begin{example} Figure~\ref{burn-example} illustrates a running of the
  $\dfs$-burning algorithm.  Vertex $i$ is labeled $v_i$.  The value of the
  parking function $\pf$ at $v_i$ is the $i$-th component of the vector
  $(0,0,1,0)$.  The root vertex, $r=0=v_0$, is lit and fire spreads along the edge
  to the highest-numbered adjacent vertex,~$v_2$.  Since there are no drops of
  water on $v_2$, i.e., since $\pf(v_2)=0$, the vertex $v_2$ is burnt and
  $(v_0,v_2)$ is added to the list of tree edges.  Similarly, the fire spreads
  from $v_2$, causing $v_4$ to be burnt and $(v_2,v_4)$ to become a tree edge.
  The fire them attempts to spread to $v_3$, but the drop of water there is used
  to dampen the edge $(v_4,v_3)$.  Backtracking to $v_2$, the fire then spreads
  to the remaining vertices.  Note that the number of dampened edges is
  $\deg\pf$.
  
  The resulting spanning tree, $\phi(\pf)$, has inversions $(v_2,v_1)$ and
  $(v_3,v_1)$, but only $(v_2,v_1)$ is a $\kappa$-inversion since the parent of
  $v_3$ in the tree is $v_2$ and $\{v_1,v_2\}$ is not an edge in the graph.  In
  accordance with Theorem~\ref{main}, we have $g - \deg\pf = 2 - 1 = 1$.  
\end{example}

\begin{figure}[ht]
\def\nz{1pt}  
\begin{tikzpicture} [scale=0.35]
  \begin{scope}[shift={(0,0)}]
    \node[draw,shape=circle,inner sep=\nz,label=above:$v_4$] (v4) at (0,5) {$0$};
    \node[draw,shape=circle,inner sep=\nz,label=right:$v_3$] (v3) at (2,3) {$1$};
    \node[draw,shape=circle,inner sep=\nz,label=left:$v_2$] (v2) at (-2,3) {$0$};
    \node[draw,shape=circle,inner sep=\nz,label=right:$v_1$] (v1) at (2,0) {$0$};
    \node[draw,shape=circle,inner sep=\nz,minimum size=12pt,label=left:$v_0$] (v0) at (-2,0) {$\ast$};
    \draw (v0)--(v1)--(v3)--(v2)--(v4)--(v3);
    \draw (v0)--(v2);
  \end{scope}
  \node at (6,1.8) {$\Longrightarrow$};
  \begin{scope}[shift={(12,0)}]
    \node[draw,shape=circle,inner sep=\nz,label=above:$v_4$] (v4) at (0,5) {$0$};
    \node[draw,shape=circle,inner sep=\nz,label=right:$v_3$] (v3) at (2,3) {$1$};
    \node[draw,shape=circle,inner sep=\nz,minimum size=12pt,label=left:$v_2$] (v2) at (-2,3) {$\ast$};
    \node[draw,shape=circle,inner sep=\nz,label=right:$v_1$] (v1) at (2,0) {$0$};
    \node[draw,shape=circle,inner sep=\nz,minimum size=12pt,label=left:$v_0$] (v0) at (-2,0) {$\ast$};
    \draw (v0)--(v1)--(v3)--(v2)--(v4)--(v3);
    \draw[->,ultra thick] (v0)--(v2);
  \end{scope}
  \node at (18,1.8) {$\Longrightarrow$};
  \begin{scope}[shift={(24,0)}]
    \node[draw,shape=circle,inner sep=\nz,minimum size=12pt,label=above:$v_4$] (v4) at (0,5) {$\ast$};
    \node[draw,shape=circle,inner sep=\nz,label=right:$v_3$] (v3) at (2,3) {$1$};
    \node[draw,shape=circle,inner sep=\nz,minimum size=12pt,label=left:$v_2$] (v2) at (-2,3) {$\ast$};
    \node[draw,shape=circle,inner sep=\nz,label=right:$v_1$] (v1) at (2,0) {$0$};
    \node[draw,shape=circle,inner sep=\nz,minimum size=12pt,label=left:$v_0$] (v0) at (-2,0) {$\ast$};
    \draw (v0)--(v1)--(v3)--(v2);
    \draw[->,ultra thick] (v2)--(v4);
    \draw (v4)--(v3);
    \draw[->,ultra thick] (v0)--(v2);
  \end{scope}
  \node[rotate=90] at (24,-2.5) {$\Longleftarrow$};
  \begin{scope}[shift={(24,-11)}]
    \node[draw,shape=circle,inner sep=\nz,minimum size=12pt,label=above:$v_4$] (v4) at (0,5) {$\ast$};
    \node[draw,shape=circle,inner sep=\nz,label=right:$v_3$] (v3) at (2,3) {$0$};
    \node[draw,shape=circle,inner sep=\nz,minimum size=12pt,label=left:$v_2$] (v2) at (-2,3) {$\ast$};
    \node[draw,shape=circle,inner sep=\nz,label=right:$v_1$] (v1) at (2,0) {$0$};
    \node[draw,shape=circle,inner sep=\nz,minimum size=12pt,label=left:$v_0$] (v0) at (-2,0) {$\ast$};
    \draw (v0)--(v1)--(v3)--(v2);
    \draw[->,ultra thick] (v2)--(v4);
    \draw[dashed,thick] (v4)--(v3);
    \draw[->,ultra thick] (v0)--(v2);
  \end{scope}
  \node at (18,-9.5) {$\Longleftarrow$};
  \begin{scope}[shift={(12,-11)}]
    \node[draw,shape=circle,inner sep=\nz,minimum size=12pt,label=above:$v_4$] (v4) at (0,5) {$\ast$};
    \node[draw,shape=circle,inner sep=\nz,minimum size=12pt,label=right:$v_3$] (v3) at (2,3) {$\ast$};
    \node[draw,shape=circle,inner sep=\nz,minimum size=12pt,label=left:$v_2$] (v2) at (-2,3) {$\ast$};
    \node[draw,shape=circle,inner sep=\nz,label=right:$v_1$] (v1) at (2,0) {$0$};
    \node[draw,shape=circle,inner sep=\nz,minimum size=12pt,label=left:$v_0$] (v0) at (-2,0) {$\ast$};
    \draw (v0)--(v1)--(v3);
    \draw[->,ultra thick] (v2)--(v3);
    \draw[->,ultra thick] (v2)--(v4);
    \draw[dashed,thick] (v4)--(v3);
    \draw[->,ultra thick] (v0)--(v2);
  \end{scope}
  \node at (6,-9.5) {$\Longleftarrow$};
  \begin{scope}[shift={(0,-11)}]
    \node[draw,shape=circle,inner sep=\nz,minimum size=12pt,label=above:$v_4$] (v4) at (0,5) {$\ast$};
    \node[draw,shape=circle,inner sep=\nz,minimum size=12pt,label=right:$v_3$] (v3) at (2,3) {$\ast$};
    \node[draw,shape=circle,inner sep=\nz,minimum size=12pt,label=left:$v_2$] (v2) at (-2,3) {$\ast$};
    \node[draw,shape=circle,inner sep=\nz,minimum size=12pt,label=right:$v_1$] (v1) at (2,0) {$\ast$};
    \node[draw,shape=circle,inner sep=\nz,minimum size=12pt,label=left:$v_0$] (v0) at (-2,0) {$\ast$};
    \draw (v0)--(v1);
    \draw[->,ultra thick] (v3)--(v1);
    \draw[->,ultra thick] (v2)--(v3);
    \draw[->,ultra thick] (v2)--(v4);
    \draw[dashed,thick] (v4)--(v3);
    \draw[->,ultra thick] (v0)--(v2);
  \end{scope}
\end{tikzpicture}
\caption{An application of the $\dfs$-burning algorithm. Tree edges are arrows,
  dampened edges are dashed, and asterisks denote burnt vertices.  To avoid
  confusion, vertex $i$ is labeled $v_i$.  The root vertex is $r=v_0$.  The
  values of the parking function are indicated on each
vertex.}\label{burn-example}
\end{figure}
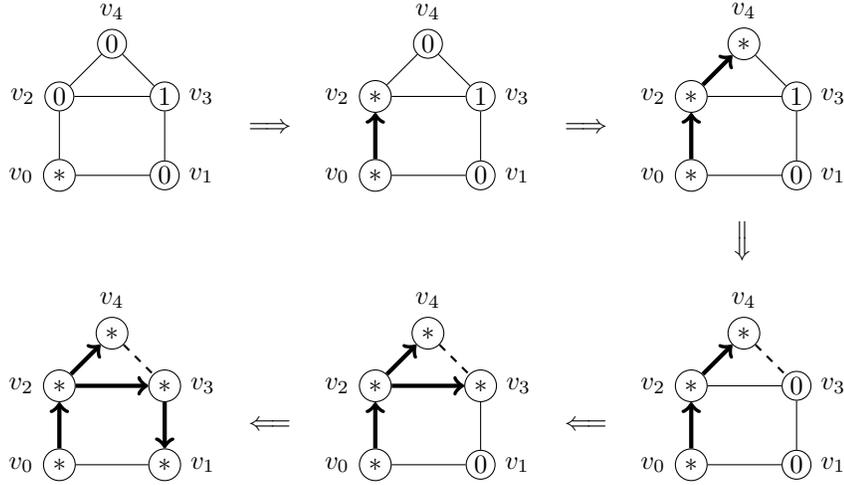
\medskip

{\small
\begin{algorithm}[ht]
\caption{ {\bf $\dfs$-burning algorithm}.}\label{dfs-burning}
\begin{algorithmic}[1]
  \Statex
  \Statex{\hspace{-0.5cm}\sc algorithm}
  \Statex{\bf Input:} $\pf\colon V\setminus\{r\}\to\NN$
  \State $\bv =\{r\}$
  \State $\damp = \{\,\}$
  \State $\tree= \{\,\}$
  \State execute {\sc dfs\_from}($r$)
  \Statex{\bf Output:} {$\bv$ and $\tree$}
  \Statex \rule{0.5\textwidth}{.4pt}
  \Statex{\hspace{-0.5cm}\sc auxillary function}
    \Function{\sc dfs\_from}{$i$} 
      \ForAll{$j$ adjacent to $i$ in $G$, from largest numerical value to smallest}
	\If{$j\notin\bv$}
	  \If{$\pf(j)=0$}
	    \State append $j$ to $\bv$
	    \State append $(i,j)$ to $\tree$
	    \State {\sc dfs\_from}$(j)$
	  \Else
	     \State $\pf(j) = \pf(j)-1$
	     \State append $(i,j)$ to $\damp$
	  \EndIf
       \EndIf	
     \EndFor
   \EndFunction
  \end{algorithmic}
\end{algorithm}
}
{\small
\begin{algorithm}[ht]
\caption{ {\bf Tree to parking function algorithm}.}\label{inverse}
\begin{algorithmic}[1]
  \Statex
  \Statex{\hspace{-0.5cm}\sc algorithm}
  \Statex{\bf Input:} Spanning tree $T$ rooted at $r$ with edges directed away
  from root.
  \State $\bv =\{r\}$
  \State $\damp = \{\,\}$
  \State $\pf=0$\ \ \ // (the $0$-function on the non-root vertices of $G$)
  \State execute {\sc tree\_from$(r)$}
  \Statex{\bf Output:} {$\pf\colon V\setminus\{r\}\to\NN$}
  \Statex \rule{0.5\textwidth}{.4pt}
  \Statex{\hspace{-0.5cm}\sc auxillary function}
  \Function{\sc tree\_from}{$i$} 
    \ForAll{$j$ adjacent to $i$ in $G$, from largest numerical value to smallest}
      \If{$j\notin\bv$}
	\If{$(i,j)$ is an edge of $T$}
	    \State append $j$ to $\bv$
	    \State {\sc tree\_from}$(j)$
	 \Else
	   \State $\pf(j) = \pf(j)+1$
	   \State append $(i,j)$ to $\damp$
	 \EndIf
      \EndIf
    \EndFor
  \EndFunction 
  \end{algorithmic}
\end{algorithm}
}
\begin{thm}[$\dfs$-bijection]\label{dfs-bijection}
After applying the $\dfs$-burning algorithm to $\pf\colon
V\setminus\{r\}\to\NN$, if\ \,{\em $\bv=V$}, then $\pf$ is a parking function
for $G$ and\ {\em $\tree$} forms a spanning tree of $G$.  If {\em $\bv\neq V$},
the nonempty set {\em $S:=V\setminus\bv$} has the property that
$\pf(j)\geq\deg_{S^c}(j)$ for all $j\in S$, certifying that $\pf$ is not a
parking function for $G$.  

Associating to each parking function the spanning tree produced by the
$\dfs$-burning algorithm defines a bijection
\[
 \phi\colon\gpf\to\spt.
\]
The inverse is provided by Algorithm~\ref{inverse}.
\end{thm}
\begin{proof}
  The $\dfs$-burning algorithm must terminate
  since the function {\sc dfs\_from()} is called at most once per vertex and
  the loop starting at line~10 then eventually considers (perhaps after
  backtracking from a later call to {\sc dfs\_from()}) each adjacent vertex
  exactly once.  Algorithm~\ref{inverse} terminates for similar reasons.
  
  After applying the $\dfs$-burning algorithm to $\pf$, if every vertex is
  burnt, then a collection of $n$ edges is returned.   These edges form a
  connected subgraph of~$G$ containing all $n+1$ vertices and hence is a
  spanning tree.  To see that in this case $\pf$ is a parking function for $G$,
  suppose $S$ is a nonempty subset of $V\setminus\{r\}$.  Suppose that $j$ is the first
  vertex of $S$ to be burnt.  Just before $j$ is added to the list of burnt
  vertices (line~13), $\pf(j)$ edges incident on $j$ will have been already
  added to the list of dampened edges.  Then, just after $j$ is burnt, a new
  edge incident on $j$ is added to the list of tree edges.  Each of these edges
  has the form $(i,j)$ where $i$ is a vertex burnt prior to $j$, and hence
  $i\notin S$.  This shows that $\pf(j)<\deg_{S^c}(j)$, as required.
  
  Now suppose that when the algorithm terminates, not every vertex is burnt.
  Let $S$ be the nonempty set of unburnt vertices.  Take $j\in S$, and consider
  the set $I$ of vertices adjacent to $j$ but not in $S$.  For each $i\in I$,
  the algorithm at some point added the edge $(i,j)$ to the set of dampened
  edges and decreased $\pf(j)$ by one (while maintaining its nonnegativity).
  Thus, $\pf(j)\geq\deg_{S^c}(j)$ (for the original input function $\pf$), which
  shows that $\pf$ is not a parking function for $G$.  

We now show that the mapping $\phi$ is bijective with inverse provided by
Algorithm~\ref{inverse}.  Consider the execution of the $\dfs$-burning algorithm
with input $\pf$.  When the for-loop at line~6 is entered, if $j$ is not a burnt
vertex, an edge $e=(i,j)$ is added either to $\tree$ or to $\damp$.  Make note
of which case occurs, and in this way, create an ordered list $L(\pf)$ of
directed edges, each edge marked as either a tree edge or dampened edge.  The
same edge may appear multiple times marked as dampened but appears at most once
as a tree edge.  Once an edge is marked as a tree edge, it will never appear
later in the list with either marking.  For each non-root vertex $j$, let
$d(\pf,j)$ denote the number of times~$j$ appears as the head of a dampened edge
in the list~$L(\pf)$. Then $\pf(j)\geq d(\pf,j)$ for all $j$ and
$\pf(j)=d(\pf,j)$ if $j$ appears as the head of a tree edge in $L(\pf)$. The
function~$\pf$ is a parking function if and only if each non-root vertex appears
as the head of tree edge in the list, in which case the edges marked as tree
edges form a spanning tree and $\pf(j)=d(\pf,j)$ for each non-root vertex~$j$.

For each $T\in\spt$ create a similar list $L(T)$ of marked edges using
Algorithm~\ref{inverse}. When a directed edge $e=(i,j)$ is considered in the
for-loop, it is either part of $T$---in which case, mark~$e$ as a tree edge---or
it is added to $\damp$, in which case, mark it as a dampened edge. In the end,
each non-root vertex appears as the head of a tree edge exactly once, and that
edge never appears subsequently in the list with either marking.  Denote the
output of the algorithm as $\pf_{T}$.  For each non-root vertex $j$, define
$d(T,j)$ as above to be the number of times $j$ appears as the head of a dampened
edge in~$L(T)$, and note that $\pf_T(j)=d(T,j)$.

For each $T\in\spt$, we claim $L(\pf_T)=L(T)$.  If not, consider the first entry
at which the two lists differ.  This entry must consist of the same edge, say
$e=(i,j)$, but with different markings.  First, suppose $e$ is marked as a tree
edge in $L(T)$ and as a dampened edge in $L(\pf_T)$.  In that case, we get the
contradiction:
\[
\pf_T(j)\geq d(\pf_T,j) > d(T,j)= \pf_T(j).
\]
Second, suppose $e$ is marked as a dampened edge in $L(T)$ and as a tree edge in
$L(\pf_T)$.  In that case, we get the contradiction:
\[
\pf_T(j)=d(\pf_T,j) < d(T,j)= \pf_T(j).
\]
Since $L(\pf_T)=L(T)$, it follows that $\phi(\pf_T)=T$ and, thus, $\pf_T$ is a
parking function.  In particular, the mapping $T\to\pf_T$ is injective.

To see that $T\to\pf_T$ is the inverse of $\phi$, it now suffices to show $\phi$
is injective.  For each $\pf\in\gpf$, we have seen that $\pf$ and $\phi(\pf)$
are determined by $L(\pf)$.  Suppose $\pf,\pf'\in\gpf$ and $\pf\neq \pf'$.  It
follows that $L(\pf)\neq L(\pf')$.
Consider the first entry in which the lists differ.  This entry is an edge $e=(i,j)$
marked as a tree edge in one list, say in $L(\pf)$, and a dampened edge in the
other, $L(\pf')$.  If follows that $e$ appears in $\phi(\pf)$ but not in
$\phi(\pf')$.  Thus, $\phi$ is injective.
\end{proof}

We proceed to a proof of our main result, Theorem~\ref{main}.
\begin{defn}
  The {\em depth-first search tree} ({\em $\dfs$-tree}) of $G$, denoted $\dfs(G)$, is the output of the
  $\dfs$-burning algorithm with input $\pf=0$.
\end{defn}

\begin{lemma}\label{main lemma} Suppose $H$ is a connected graph obtained by
  deleting an edge of $\dfs(G)$ from $G$.  Then the $\kappa$-inversions of
  $\dfs(H)$ as a subgraph of $G$ are the same as those as a subgraph of $H$
  (assuming the same root for both $G$ and $H$), and
  \[
\kappa(G,\dfs(H))=\kappa(H,\dfs(H))=g-1,
  \]
  where $g=|E|-|V|+1=|E|-n$.
\end{lemma}
  \begin{proof} Let $T=\dfs(H)$.  Trivially, $\kappa$-inversions of $T$ in $H$
    are $\kappa$-inversions for $T$ in $G$.  We prove the opposite inclusion by
    contradiction.  Suppose $(i,j)$ is a $\kappa$-inversion for $T$ in $G$ but
    not in $H$.  In other words, letting $i'$ be the parent of $i$ in $T$, the
    edge $e=\{i',j\}$ is in $G$ but not in $H$, hence, $e$ is the edge of
    $\dfs(G)$ deleted to obtain $H$.  

  Since $G$ and $H$ differ only in the edge $e$, the depth-first searches of
  both $G$ and $H$ are the same up to the point at which the vertex $i'$ of $e$
  is reached.  Next, since $i>j$, the depth-first search of $G$ travels from
  $i'$ to $i$, i.e., $(i',i)$ is an edge of $\dfs(G)$.  Subsequently, the path
  from $i$ to $j$ in $T$ must also be eventually added to $\dfs(G)$.  Hence, $e$
  cannot be in $\dfs(G)$, which is a contradiction.  

  To show that $\kappa(H,\dfs(H))=g-1$, suppose $e=\{i',j\}$ is an edge of $H$
  but not an edge of $\dfs(H)$.  Without loss of generality, assume $i'$ is
  added to the list of burnt vertices before $j$ in the execution of the
  $\dfs$-burning algorithm used to create $\dfs(H)$.  Since $\{i',j\}$ is not in
  $\dfs(H)$, when the algorithm burns $j$, it has not yet backtracked to $i'$.
  Hence, there is a path in $\dfs(H)$ from $i'$ to $j$ of edges directed away
  from the root.  If $i$ is the child of $i'$ in this path, then $(i,j)$ is a
  $\kappa$-inversion for $\dfs(H)$.  In this way, we get a bijection between
  edges of $H$ that are not edges of $\dfs(H)$ and $\kappa$-inversions of
  $\dfs(H)$.  The result follows. 
\end{proof}

\begin{proof}[\it Proof of Theorem~\ref{main}]  Let $\pf\in\gpf$.  It remains to
  be shown that $\kappa(G,\phi(\pf))=g-\deg\pf$.
  
  Let $D=\{e_1,\dots,e_k\}$ be the dampened edges resulting from applying the
  $\dfs$-burning algorithm to~$\pf$.  We assume that these edges are listed in
  the order they were found by the algorithm, and note that $k=\deg\pf$.  Define
  $G_0:=G$, and for $\ell=1,\dots, k$, let $G_{\ell}$ be the graph obtained from
  $G_{\ell-1}$ by removing edge~$e_{\ell}$.  Each $G_i$ contains $\phi(\pf)$ and
  is consequently connected.  Further, for $G_k$, obtained from $G$ by removing
  all the dampened edges, we have $\dfs(G_k)=\phi(\pf)$. 
  
  We now show that $e_{\ell}\in\dfs(G_{\ell-1})$ for $0\leq \ell<k$, from which
  the result follows by repeated application of Lemma~\ref{main lemma}.  The key
  idea is that starting with {\em any} connected simple graph and {\em any}
  nonnegative function on the graph's non-root vertices, the first dampened edge
  created by the $\dfs$-burning algorithm is an edge in the graph's depth-first
  search tree.  For example, $e_1$ is an edge of $\dfs(G)=\dfs(G_0)$.  Define
  $\pf_0=\pf$, and for $1\leq\ell<k$, define $\pf_{\ell}\colon
  V\setminus\{r\}\to\NN$ by
\begin{align*}
  \pf_{\ell}(j) = 
 \begin{cases}
   \pf_{\ell-1}(j)-1&\text{if $e_{\ell}=(i,j)$},\\
   \pf_{\ell-1}(j)&\text{otherwise}.\\
 \end{cases}
\end{align*}
Then $e_{\ell+1}$ is in $\dfs(G_{\ell})$ since it is the first dampened edge when
the algorithm is run with input $G_{\ell}$ and $\pf_{\ell}$ (the full sequence
of dampened edges being $e_{\ell+1},\dots,e_k$).  
\end{proof}

\begin{cor} $\kappa(G,\dfs(G))=g$.
\end{cor}
\begin{proof} Apply Theorem~\ref{main} to $\pf=0$.
\end{proof}

\section{Threshold graphs}\label{threshold graphs} Threshold graphs are a family of graphs, including
the complete graphs, introduced by Chvatal and Hammer~\cite{Chvatal}.  For a
comprehensive study, see~\cite{Mahadev}.  We show that if properly labeled,
there is no distinction between inversions and $\kappa$-inversions for their
spanning trees.

\begin{defn} A graph is a {\em threshold graph} if it can be constructed from
  a graph with one vertex and no edges by repeatedly carrying out the following two steps:
\begin{itemize}
  \item Add a {\em dominating vertex}: a vertex that is connected to every other
    existing vertex.
  \item Add an {\em isolated vertex}: a vertex that is not connected to any
    other existing vertex.
\end{itemize}
\end{defn}

A threshold graph with more than one vertex is connected if and only if the
last-added vertex is dominating.  Each threshold graph is uniquely defined by
its {\em build sequence}: a string starting with the symbol $\ast$ (for the
initial vertex) followed by any string consisting of the letters {\bf d} (for
the addition of a dominating vertex) and {\bf i} (for the addition of an
isolated vertex).  Thus, $\ast${\bf iddid} describes the threshold graph
pictured in Figure~\ref{threshold} formed from a single vertex by adding, in
order, an isolated vertex, two dominating vertices, an isolated vertex, then a
final dominating vertex.
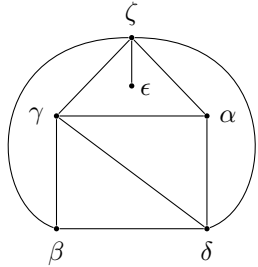
\begin{figure}[ht]
\begin{tikzpicture} [my_node/.style={fill, circle, inner sep=0pt,minimum size=2pt}, scale=0.5]
\centering                                 
\draw (0.37,3.7) node {$\epsilon$};
\node[my_node] (v5) at (0,3.8) {};
\node[my_node,above,label=above:$\zeta$] (v4) at (0,5) {};
\node[my_node,label=right:$\alpha$] (v3) at (2,3) {};
\node[my_node,label=left:$\gamma$] (v2) at (-2,3) {};
\node[my_node,label=below:$\delta$] (v1) at (2,0) {};
\node[my_node,label=below:$\beta$] (v0) at (-2,0) {};
\draw (v0)--(v1)--(v3)--(v2)--(v4)--(v3);
\draw (v1)--(v2);
\draw (v0)--(v2);
\draw (v4) .. controls +(180:4) and +(160:2) .. (v0);
\draw (v4) .. controls +(0:4) and +(20:2) .. (v1);
\draw (v4)--(v5);
\end{tikzpicture}
\caption{The threshold graph with build sequence $\ast${\bf iddid}.  The vertices are
labeled $\alpha,\beta,\gamma,\delta,\epsilon,\zeta$, in the order of the build
sequence.  }\label{threshold}.
\end{figure}
Omitting $\ast$ and reading left-to-right, group the consecutive sequences
consisting entirely of a single letter (either {\bf d} or {\bf i}) into {\em
blocks}, then include $\ast$ in the first block.  Thus, the sequence of blocks
for $\ast${\bf iddid} is [$\ast${\bf i}], [{\bf dd}], [{\bf i}], [{\bf d}]. 

We say that a threshold graph is {\em labeled by reverse degree sequence} if its
vertices are labeled by $0,\dots, n$ in such a way that $\deg(i)\geq\deg(j)$ for
each pair of vertices $i<j$.  If there is more than one vertex with the same
degree, the labeling is not unique.  For example, for the graph in
Figure~\ref{threshold}, (i)~$\zeta$ must be labeled $0$, (ii) $\gamma,\delta$
must be labeled $1,2$ in either order, (iii) $\alpha, \beta$ must be labeled
$3,4$ in either order, and (iv) $\epsilon$ must be labeled $5$.

\begin{prop}\label{threshold-inversion} Let $G$ be a connected threshold graph labeled by reverse degree
  sequence, and let $T$ be a spanning tree of $G$.  Then every inversion of $T$
  is a $\kappa$-inversion.
\end{prop}
\begin{proof} First note that two vertices have the same degree in $G$ if and
  only if they belong to the same block of the build sequence for $G$. The
  degree of any vertex in a {\bf d}-block is greater than the degree of any
  vertex in an {\bf i}-block.  Also, block-by-block, the degrees for vertices in
  successive {\bf d}-blocks increase from left-to-right, and the degrees for
  {\bf i}-blocks decrease.   
  
  Let $(i,j)$ be an inversion of $T$, and let $i'$ be the parent of $i$.  We
  must show that $\{i',j\}\in E$, where $E$ denotes the set of edges of $G$.
  Since $i>j$ and $G$ is labeled by reverse degree sequence,
  $\deg(i)\leq\deg(j)$.  If $\deg(i)=\deg(j)$, then $i$ and $j$ belong to the
  same block in the build sequence for $G$.  In that case, not counting each
  other, $i$ and $j$ have the same neighbors.  Hence, $\{i',j\}\in E$.
  Otherwise, $\deg(i)<\deg(j)$, and the result follows using the build sequence
  and considering cases:
  \begin{itemize}
    \item If $i$ and $j$ are both dominating, then $\deg(i)<\deg(j)$ implies $j$
      follows $i$ in the build sequence, and thus  $i'$ must be adjacent to $j$.
    \item If $i$ is isolated and $j$ is dominating, then $i'$ must be
      dominating.  So $\{i',j\}\in E$ since every pair of dominating vertices are
      adjacent in $G$.
    \item If both $i$ and $j$ are isolated, then $j$ precedes $i$ in the build
      sequence and $i'$ is a dominating vertex added after $i$.  Hence,
      $\{i',j\}\in E$.
    \item It is not possible for $i$ to be dominating and $j$ to be isolated
      since $\deg(i)<\deg(j)$.
  \end{itemize}
\end{proof}

\begin{rmk}  In light of Proposition~\ref{threshold-inversion}, for a threshold
  graph one may replace $\kappa(G,\phi(\pf))$ in Theorem~\ref{main} by the
  number of inversions of $\phi(\pf)$.  Any labeling of a complete graph is by
  reverse degree sequence since the degree sequence in question is constant.
  Hence, Proposition~\ref{threshold-inversion} applies and shows that
  specializing Theorem~\ref{main} to the case of complete graphs gives a
  solution to a problem posed by Stanley~\cite[Chapter 6,~Exercise 4]{Stanley}.
\end{rmk}
\newpage 

\bibliographystyle{plain}
\bibliography{b5.bib}

\end{document}